\documentclass{amsart}
\usepackage{amsmath,amsthm,amssymb}

  \newtheorem{thm}{Theorem}[section]

  \newtheorem{lem}{Lemma}[section]

\numberwithin{equation}{section}
  
\newcommand{\bsquare}{\hbox{\rule{6pt}{6pt}}}  

\newcommand{\mapright}[1]{%
  \smash{\mathop{%
    \hbox to 1cm{\rightarrowfill}}\limits^{#1}}}
    \newcommand{\mapleft}[1]{%
  \smash{\mathop{%
    \hbox to 1cm{\leftarrowfill}}\limits_{#1}}}

\begin{document}

\title{On extendibility of a map induced by Bers isomorphism}


\author{Hideki Miyachi}
\address{Department of Mathematics,
Graduate School of Science, Osaka University,
Machikaneyama 1-1, Toyonaka, Osaka, 560-0043, Japan}
\curraddr{}
\email{miyachi@math.sci.osaka-u.ac.jp}
\thanks{The first author is partially supported by the Ministry of Education, Science, Sports and Culture, Grant-in-Aid for Scientific Research (C), 21540177.}

\author{Toshihiro Nogi}
\address{Osaka City University Advanced Mathematical Institute, 
Sugimoto, Sumiyoshi-ku Osaka 558-8585, Japan}
\email{nogi@sci.osaka-cu.ac.jp}
\thanks{The second author is partially supported by the JSPS Institutional Program for Young Research Overseas Visits
``Promoting international young researchers in mathematics and mathematical sciences led by OCAMI ".}


\date{}

\dedicatory{}

\begin{abstract}
Let $S$ be a closed Riemann surface of genus $g(\geqq 2)$ and
set $\dot{S}=S \setminus \{ \hat{z}_0 \}$.
Then we have the composed map $\varphi\circ r$ of 
 a map $r: T(S) \times U \rightarrow F(S)$ and the Bers isomorphism $\varphi: F(S)  \rightarrow T(\dot{S})$,
where $F(S)$ is the Bers fiber space of $S$, $T(X)$ is the Teichm\"uller space of $X$ and $U$ is the upper half-plane.

The purpose of this paper is to show the map 
$\varphi\circ r:T(S)\times U \rightarrow T(\dot{S})$. 
has a continuous extension to some subset of the boundary $T(S) \times \partial U$. 
 \end{abstract}

\maketitle

\section{Introduction}
\subsection{Teichm\"uller space} Let $S$ be a closed Riemann surface of genus $g(\geqq 2)$.
Consider any pair $(R, f)$ of a closed Riemann surface $R$ of genus $g$ and 
a quasiconformal map $f:S \rightarrow R$. 
Two pairs $(R_1, f_1)$ and $(R_2, f_2)$ are said to be {\it equivalent} 
if $f_2\circ f_1^{-1}:R_1 \rightarrow R_2$ is homotopic to a biholomorphic map $h : R_1  \rightarrow R_2$. 
Let $[R, f]$ be the equivalence class of such a pair $(R, f)$. 
We set $$T(S) = \{ [R, f] ~|~f : S \rightarrow R : \text{quasiconformal} \}$$
and call $T(S)$ the {\it Teichm\"uller space} of $S$.

For any $p_1=[R_1,f_1]$, $p_2=[R_2,f_2]\in T(S)$, the \emph{Teichm\"uller
distance} is defined to be
$$
d_T(p_1,p_2)=\frac{1}{2}\inf_g \log K(g)
$$
where $g$ runs over all qusiconformal maps from $R_1$ to $R_2$
homotopic to $f_2\circ f_1^{-1}$ and $K(g)$ means the maximal dilatation of $g$.
The Teichm\"uller space is topologized with the Teichm\"uller distance. 

It is known that $S$ can be represented as $U/G$ 
where $U$ is the upper half-plane 
and $G$ is a torsion free Fuchsian group. 
Let $L_{\infty}(U, G)_1$ be the space 
of measurable functions $\mu$ on $U$ satisfying 
\begin{itemize}
\item[(1)] $\| \mu \|_{\infty} = \sup_{z \in U}|\mu(z)| <1,$
\item[(2)] $(\mu\circ g)\dfrac{\overline{g}'}{g}$ for all $g \in G$.
\end{itemize}
For any $\mu \in L_{\infty}(U, G)_1$, there is a unique quasiconformal map $w$ of $U$ onto $U$ 
satisfying normalization conditions $w(0)=0, w(1)=1$ and $w(\infty)=\infty$.  
Let $Q(G)$ be the be the set of all normalized quasiconformal map $w$ 
such that $wGw^{-1}$ is also Fushsian. 
We write $w=w_{\mu}.$  
Two maps $w_1,w_2 \in Q(G)$ are said to be {\it equivalent} 
if $w_1 = w_2$ on the real axis $\mathbb{R}$. 
Let $[w]$ be the equivalence class of $w \in Q(G) $.   
We set $$T(G)=\{ [w] ~|~ w \in Q(G) \}$$
and call $T(G)$ the {\it Teichm\"uller space} of $G$. 

Then we have a canonical bijection
\begin{equation}
\label{eq:canonical_bijection}
T(G)\ni [w_{\mu}]\mapsto [U/G_{\mu}, f_{\mu}] \in T(S)
\end{equation}
where $G_{\mu}=w_{\mu}G w_{\mu}^{-1}$ 
and $f_{\mu}$ is the map induced by
$w_{\mu} :U \rightarrow U$.
Throughout this paper, 
we always identify $T(G)$ with $T(S)$ 
via the bijection \eqref{eq:canonical_bijection}.

\subsection{Bers fiber space} For any $\mu \in L_{\infty}(U, G)_1$, 
there is a unique quasiconformal map $w^{\mu}$ of $\hat{\mathbb{C}}$ 
with $w^{\mu}(0)=0, w^{\mu}(1)=1, w^{\mu}(\infty)=\infty$  
such that $w^{\mu}$ satisfies the Beltrami equation $w_{\bar{z}}=\mu w_{z}$ on $U$, 
and is conformal on the lower half-plane $L$.  
The {\it Bers fiber space} $F(G)$ over $T(G)$ is defined by
$$
F(G)=\{([w_{\mu}], z) \in T(G) \times \hat{\mathbb{C}} ~|~[w_{\mu}] \in T(G),~ z \in w^{\mu}(U) \}.
$$ 

Take a point $z_0 \in U$ 
and denote by $A$ the set of all points $g(z_0)$, 
$g \in G$.   
Let 
$$
v : U \rightarrow U-A
$$
be a holomorphic universal covering map. 
We define 
$$
\dot{G}=\{ h \in {\rm Aut} ~U~|~v\circ h=g\circ v ~\text{for some}~g \in G~\}.
$$
We see that $U/\dot{G}=U/G-\{\pi (z_0)\}$, 
where $\pi : U \rightarrow S=U/G$ is the natural projection.  
Set $\dot{S}=U/\dot{G}$. 
By Lemma 6.3 of Bers \cite{bers}, 
every point in $F(G)$ is represented 
as a point $([w_{\mu}], w^{\mu}(z_0))$ 
for some $\mu \in L_{\infty}(U, G)_1$. 
For $\mu \in L_{\infty}(U, G)_1$, 
we define $\nu \in L_{\infty}(U, \dot{G})_1$ by
$$
\mu(v(z)) \frac{\overline{v'(z)}}{v'(z)} = \nu(z).
$$
Then, Bers' isomorphism theorem asserts that the map
$$
\varphi : ([w_{\mu}], w^{\mu}(z_0)) \mapsto [w_{\nu}]
$$
is a biholomorphic bijection map (cf. Theorem 9 of \cite{bers}).
Moreover 
we define a map $r : T(G) \times U \rightarrow F(G)$ 
by
$$
([w_{\mu}], z) \mapsto ( [w_{\mu}], h_{ [ w_{\mu} ] }(z) ).
$$
where $U$ is the universal covering of $S$ 
and $h_{ [ w_{\mu} ] } :U \rightarrow w^{\mu}(U)$ is the Teichm\"uller mapping in 
the class of $w^{\mu}$.
We remark that our definition of $r$ is different from Bers' one. 
See the proof of Lemma 6.4 of \cite{bers}.
This map $r$ is not real analytic, 
but it is a homeomorphism.  
This difference does not influence our purpose.  

Via the bijection \eqref{eq:canonical_bijection}, 
the Bers fiber space $F(S)$ over $T(S)$ is defined by
$$
F(S)=\{([R_{\mu},f_{\mu}], z) \in T(S) \times \hat{\mathbb{C}} ~|~[R_{\mu},f_{\mu}] \in T(S),~ z \in w^{\mu}(U)  \}
$$
with the projection
$$
F(S) \ni ([R_{\mu},f_{\mu}], z)  \mapsto [R_{\mu},f_{\mu}] \in T(S).
$$

Similarly, 
we have the isomorphism $F(S) \rightarrow T(\dot{S})$ and 
the homeomorphism $T(S) \times U \rightarrow F(S)$, 
and we denote them by the same symbols $\varphi$ and $r$, respectively. 

\subsection{The Bers embedding} The Teichm\"uller space $T(S)$ can be regarded canonically 
as a bounded domain of a complex Banach space $B_2(L, G)$ 
in the following way:    
Let $B_2(L, G)$ consist of all holomorphic functions $\phi$ defined on $L$ 
such that
$$
\phi(g(z))g'(z)^2=\phi(z) ~\text{for}~g\in G ~\text{and}~  z\in L
$$
and
$$
\|\phi \|_{\infty} = \sup_{z \in L} |({\rm Im}z)^2\phi (z)|<\infty.
$$
For any $\mu \in L_{\infty}(U, G)_1$, 
we denote by $\phi^{\mu}$ the Schwarzian derivative of $w^{\mu}$ on $L$, that is, 
$$
\phi^{\mu}(z) = \{w^{\mu},z \}=\frac{(w^{\mu})'''(z)}{(w^{\mu})'(z)}-\frac{3}{2} \left( \frac{(w^{\mu})''(z)}{(w^{\mu})'(z)} \right)^2  ~~\text{for}~~ z\in L.
$$
If $\mu \in L_{\infty}(U, G)_1$, 
then $\phi^{\mu} \in B_2(L, G)$ and 
the {\it Bers embedding} $T(S) \ni [R_\mu,f_\mu] \mapsto \phi^{\mu} \in B_{2}(L, G)$ is a biholomorphic bijection 
of $T(S)$ 
onto a holomorphically  bounded domain in $B_2(L, G)$. 
From now on, 
we will identify $T(S)$ 
with its image in $B_2(L, G)$.

Similarly, we define the Bers embedding of $T(\dot{S})$ 
into $B_{2}(L, \dot{G})$.  
Since $F(S)$ is a domain of $B_2(L, G) \times \hat{\mathbb{C}}$ 
and $T(\dot{S})$ is a bounded domain in $B_2(L, \dot{G})$, 
we define the topological boundaries of them naturally. 
Let $\overline{F(G)}$ denote the closure of $F(G)$ 
in $B_2(L, G) \times \hat{\mathbb{C}}$.  

\subsection{Main theorem} Zhang \cite{zhang} proved 
the Bers isomorphism $\varphi$ cannot be continuously extended 
to $\overline{F(S)}$ 
if the dimension of $T(S)$ is greater than zero. 
Then we have the following question:   
Is there some subset of $\overline{F(S)}-F(S)$ 
to which $\varphi$ can be continuously extended ?

To do this, 
we compose the isomorphism $\varphi : F(S) \rightarrow T(\dot{S})$ 
and the map $r: T(S)\times U \rightarrow F(S)$, 
then we obtain new map $\varphi \circ r:T(S) \times U \rightarrow T(\dot{S})$. 
Let ${\mathbb A}$ be a subset of $\partial U$ consisting of all points filling $S$ (cf. \S 3.3). 
Our main theorem is as follows.

\vspace{3mm}

{\bf Theorem}~\ref{main}
{\it The map $\varphi \circ r : T(S) \times U \rightarrow T(\dot{S})$ has a continuous extension  
to $T(S) \times {\mathbb A}$. 
}

\vspace{3mm}

The idea of proof of Theorem \ref{main} is as follows.
For any sequence $\{ ( p_m, z_m )\}_{m=1}^{\infty}$  
in $T(S) \times U$ converging to $(p_{\infty}, z_{\infty}) \in T(S) \times {\mathbb A}$,  
we put $q_m = \varphi \circ r(p_m, z_m) \in T(\dot{S})$. 
We need to prove that the sequence $\{ q_m \}_{m=1}^{\infty}$ converges  
without depending on the choice of a convergent sequence to $(p_{\infty}, z_{\infty}) \in T(S)  \times {\mathbb A}$.

Let $q_0$ be the basepoint of $T(\dot{S})$. 
It is known that 
the image of the Bers embedding is canonically identified 
with the slice $T(\dot{S}) \times \{\overline{q}_0\}$ in the quasifuchsian space
which is biholomorphic to $T(\dot{S}) \times T(\overline{\dot{S}})$ 
(cf. Chapter 8 of Bers \cite{bers1}). 
For each pair $(q_m,\overline{q}_0) \in T(\dot{S}) \times T(\overline{\dot{S}})$, 
there is a unique quasifuchsian group $\Gamma_m$ up to conjugation 
such that the conformal boundaries of 
a hyperbolic manifold $N_m={\mathbb H}^3/\Gamma_m$ correspond 
to the pair $(q_m, \overline{q}_0)$.

We assume throughout the paper that quasifuchsian groups $\Gamma_m$ and 
manifolds $N_m$ are marked by a homomorphism and homotopy equivalence,
respectively.

For our purpose, 
it is sufficient to show that a limit $\Gamma_{\infty}$ of the sequence $\{ \Gamma_m \}_{m=1}^{\infty}$ is uniquely determined. 
To do this, 
we show the following key lemma.  

\vspace{3mm}

{\bf Lemma}~ \ref{lem-1}
{\it Given $z_{\infty} \in {\mathbb A}$,  
there exists a filling lamination $\lambda$ with the following property. 
For any sequence $\{ z_m \}_{m=1}^{\infty}$ with 
$\lim_{m \to \infty} z_m = z_{\infty}$ and $q_m =\varphi \circ r(p_m, z_m)$ 
as above, there exists a sequence of simple closed curves 
$\{ \alpha_m \}_{m=1}^{\infty}$ with the following properties:
\begin{itemize}
\item[(1)] The lengths $\ell_{ N_m }( \alpha_m )$ of $\alpha_m$ in $N_m$ are bounded, and
\item[(2)] the sequence $\{ \alpha_m \}_{m=1}^{\infty}$ converges to $\lambda$ in $\overline{\mathcal{C}}(\dot{S})$. 
\end{itemize}
}

\vspace{3mm}

Here the definition of $\overline{\mathcal{C}}(\dot{S})$ will be given in \S 2 and \S 3. 
We remark that $\lambda$ is identified with an ending lamination by Klarreich's work in \cite{klarreich}. 

From this lemma, we see that the limit $\Gamma_{\infty}$ of 
$\{ \Gamma_m \}_{m=1}^{\infty}$ is singly degenerate Kleinian group, 
that is, the the region of discontinuity of $\Gamma_{\infty}$ is simply connected. 
Then by using Ending lamination theorem for surface groups of \cite{bcm}, 
$\Gamma_{\infty}$ is uniquely determined by $(\lambda, \overline{q}_{0})$ up to conjugation, 
and it is the only possible limit. 

\section{Gromov-hyperbolic spaces}

In this section, we shall give the boundary 
at infinity of hyperbolic space.
For details, see Klarreich \cite{klarreich}.

Let $(\Delta,d)$ be a metric space. 
If $\Delta$ is equipped 
with a basepoint $0$, 
we define the {\it Gromov product} $\langle x | y \rangle$ of  points $x$ and $y$ in $\Delta$ by
$$
\langle x | y \rangle = \langle x | y \rangle_0 =\frac12\{ d(x, 0)+d(y, 0)-d(x, y) \}.
$$
For $\delta \geqq 0$, the metric space $\Delta$ is said to be {\it $\delta$-hyperbolic} if
$$
\langle x | y \rangle \geqq \min\{ \langle x | z \rangle, \langle y | z \rangle \}-\delta
$$
holds for every $x, y, z \in \Delta$.
We say that $\Delta$ is {\it hyperbolic in the sense of Gromov} 
if $\Delta$ is $\delta$-hyperbolic for some $\delta \geqq 0$. 

If $\Delta$ is a hyperbolic space, 
we can define a boundary of $\Delta$ 
in the following way:   
We say that a sequence $\{x_n\}_{n=1}^{\infty}$ of points 
in $\Delta$ {\it converges at infinity} 
if it satisfies 
$\lim_{m, n \to \infty}\langle x_m | x_n \rangle= \infty$. 
Given two sequences $\{ x_n \}_{n=1}^{\infty}$ 
and $\{ y_n \}_{n=1}^{\infty}$ that converge at infinity, 
they are called to be {\it equivalent} 
if $\lim_{m, n \to \infty}\langle x_m | y_n \rangle= \infty$. 
Since $\Delta$ is a hyperbolic, 
we see that this is an equivalence relation ($\sim$). 
We set
$$
\partial_{\infty} \Delta = \{ \{x_n\}_{n=1}^{\infty} ~|~\{x_n\}_{n=1}^{\infty} \text{converges at infinity}\}/\sim
$$
and call $\partial_{\infty} \Delta$ the {\it boundary at infinity} of $\Delta$. 
If $\xi \in \partial_{\infty} \Delta$, 
then we say that a sequence of points in $\Delta$ {\it converges to $\xi$} 
if the sequence belongs to the equivalence class $\xi$. 
We put 
$$
\overline{\Delta}=\Delta \cup \partial_{\infty} \Delta.
$$  
\section{Leininger, Mj and Schleimer's work}

\subsection{The Curve Complex} 
Let $S=U/G$ be a closed Riemann surface of genus $g(\geqq 2)$ and 
$\pi : U\rightarrow S$ be the natural projection.  
We take a point $z_0$ in $U$ and 
set $\hat{z}_0 = \pi(z_0)$. Put $\dot{S} = S \setminus \{\hat{z}_0 \}$. 

The curve complex ${\mathcal C}(S)$ is a simplicial complex 
which is defined as follows.  
The vertices of ${\mathcal C}(S)$ are homotopy classes of 
nontrivial simple closed curves on $S$. 
Two curves are connected by an edge 
if they can be realized disjointly on $S$, and 
in general a collection of curves spans a simplex 
if the curves can be realized disjointly on $S$. 
We define ${\mathcal C}(\dot{S})$ similarly, with vertices   
consisting of nontrivial, non-peripheral simple closed curves on $\dot{S}$.    

We give ${\mathcal C}(S)$(resp ${\mathcal C}(\dot{S}))$ 
a metric structure by making every simplex a regular Euclidean simplex 
whose edges have length 1, and define the distance $d_{{\mathcal C}(S)}$(resp $d_{{\mathcal C}(\dot{S})}$) by taking shortest paths.

\begin{thm} [Masur and Minsky \cite{mm1}, Theorem 1.1]
The spaces ${\mathcal C}(S)$ and ${\mathcal C}(\dot{S})$ are $\delta$-hyperbolic for some $\delta>0$. 
\end{thm} 
We put $\overline{{\mathcal C}}(S) = {\mathcal C}(S) \cup \partial_{\infty} {\mathcal C}(S)$ 
and $\overline{{\mathcal C}}(\dot{S}) = {\mathcal C}(\dot{S}) \cup \partial_{\infty} {\mathcal C}(\dot{S})$, respectively. 

\subsection{Definition of $\Phi$} 
\label{subsec:definition_Phi}

Denote by ${\rm Diff}^{+}(S)$ 
the group of all orientation preserving diffeomorphisms of $S$ onto itself. 
Let ${\rm Diff}_0(S)$ be a group which consists of all elements in 
${\rm Diff}^{+}(S)$ isotopic to the identity map $id$.

We define the evaluation map
$$
{\rm ev} : {\rm Diff}^{+}(S) \rightarrow S
$$
by ${\rm ev}(f)=f(\hat{z}_0)$. 
A theorem of Earle and Eells asserts that 
${\rm Diff}_0(S)$ is contractible. 
Hence, for the map ${\rm ev}|{\rm Diff}_0(S)$, 
there is a unique lift 
$$\tilde{\rm ev} : {\rm Diff}_0(S) \rightarrow U$$
satisfying the condition that $\tilde{\rm ev}(id) = z_0$. 
 
Following Leininger, Mj and Schleimer \cite{lms}, 
we will define a map 
$\tilde{\Phi} : {\mathcal C}(S) \times {\rm Diff}_0(S) \rightarrow {\mathcal C}(\dot{S})$. 
To give an idea of the definition of $\tilde{\Phi}$, 
we consider the case of  ${\mathcal C}^{0}(S) \times {\rm Diff}_0(S)$ 
where ${\mathcal C}^{0}(S)$ is 0-skeleton of ${\mathcal C}(S)$.  
Take a point $(v,f) \in {\mathcal C}^{0}(S) \times {\rm Diff}_0(S)$. 
From now on, if no confusion is possible, 
we identify the homotopy class $v$ with the geodesic representative. 
Then there is an isotopy $f_t,~t\in [0, 1]$, 
between $f_0=id$ and $f_1=f$. 
Setting $C(t)=f_t(\hat{z}_0)$ for every $t \in [0,1]$, 
we have a path $C$ from $\hat{z}_0$ to $f(\hat{z}_0)$ on $S$. 
Move a point in $S$ from $f(\hat{z}_0)$ to $\hat{z}_0$ along $C$ 
and drag $v$ back along the moving point. 
Then we obtain new simple closed curve on $\dot{S}$ and 
denote the curve by $f^{-1}(v)$. 
Thus we define $\tilde{\Phi}(v,f)=f^{-1}(v)$.

However, when $f(\hat{z}_0)\in v$, 
we can not define $\tilde{\Phi}(v,f)$ as above. 
We solve this problem in the following way:   
Now choose $\{\epsilon(v) \}_{v \in {\mathcal C}^{0}(S) } \subset {\mathbb R}_{>0}$ 
so that the $\epsilon(v)$-neighborhood $N(v)=N_{\epsilon(v)}$ of $v$ 
has the following properties:
\begin{itemize} 
\item[(i)] $N(v)$ is homeomorphic to $S^1 \times [0,1]$ 
\item[(ii)] $N(v_1) \cap N(v_2)=\emptyset$ if $v_1 \cap v_2 = \emptyset$.  
\end{itemize} 
Let $N^{\circ}(v)$ be the interior of $N(v)$ 
and  $v^{\pm}$ the boundary components of $N(v)$.
Notice that $\epsilon(v)$ is depending only on 
the length of the geodesic representative
of $v$ (cf. \cite{Buser}).

If $v\subset {\mathcal C}(S)$ is a simplex 
with vertices $\{v_0, v_1, \cdots, v_k \}$, 
then we consider the barycentric coordinates 
for points in $v$:
$$\{ \sum_{j=0}^{k} s_j v_j~|~\sum_{j=0}^{k} s_j =1 ~\text{and}~  s_j\geq 0,~\text{for}~j=0,1,\cdots,k \}$$

For a point $(v,f)$ with $v$ a vertex of ${\mathcal C}(S)$, 
we can define $\tilde{\Phi}$ as follows.   
If $f(\hat{z}_0) \not\in N^{\circ}(v)$, 
then we define $$\tilde{\Phi}(v,f) = f^{-1}(v)$$ as above.

If $f(\hat{z}_0) \in N^{\circ}(v)$, then
$f^{-1}(v^{+})$ and $f^{-1}(v^{-})$ are not isotopic in $\dot{S}$. 
We set $$t = \frac{d( v^{+}, f(\hat{z}_0) ) }{2\epsilon(v)},$$
where $d(v^{+},f(\hat{z}_0))$ is the distance inside $N(v)$ 
from $f(\hat{z}_0)$ to $v^{+}$. 
Then we define
$$\tilde{\Phi}(v, f) = tf^{-1}(v^{+})+(1-t)f^{-1}(v^{-})$$
in barycentric coordinates on the edge $[f^{-1}(v^{+}),f^{-1}(v^{-})]$. 

In general, for a point $(x,f)\in {\mathcal C}(S) \times {\rm Diff}_0(S)$ 
with $x = \sum_{j=0}^{k} s_j v_j$, 
we define $\tilde{\Phi}(x, f)$ as follows. 
If $f(\hat{z}_0) \not\in \bigcup_{j=0}^{k} N^{\circ}(v_j)$, 
then we define $$\tilde{\Phi}(x,f) = \sum_{j}s_jf^{-1}(v_j).$$

If $f(\hat{z}_0) \in N^{\circ}(v_i)$ for exactly one $i$,  
we set $$t = \frac{d( v^{+}, f(\hat{z}_0) ) }{2\epsilon(v_i)},$$
and define
\begin{equation}
\label{tilde-Phi}
\tilde{\Phi}(x, f) = s_i (tf^{-1}(v_i^{+})+(1-t)f^{-1}(v_i^{-}))+\sum_{j \ne i} s_j f^{-1}(v_j).
\end{equation}

Finally, by Proposition 2.2 in \cite{lms}, 
if $\tilde{\rm ev}(f_1) = \tilde{\rm ev}(f_2) $ in $U$, 
then we see that $\tilde{\Phi}(x, f_1) = \tilde{\Phi}(x, f_2)$. 
From this, we have a map $\Phi : {\mathcal  C}(S) \times U \rightarrow {\mathcal C}(\dot{S})$ satisfying
$\tilde{\Phi} = \Phi \circ(id \times \tilde{\rm ev})$. 

\subsection{Extendibility of $\Phi$}
A subsurface of $S$ is said to be an {\it essential} if it is either a component of the complement of a
geodesic multicurve in $S$, the annular neighborhood $N(v)$ of some geodesic 
$v \in {\mathcal C}^{0}(S)$, or else $S$.

Given an essential subsurface $Y$, if a point $x \in \partial U$ has the following properties,   
\begin{itemize}
\item[(i)] for every geodesic ray $r \subset U$ ending at $x$ 
and for every $v \in {\mathcal C}^{0}(S)$ 
which nontrivially intersects an essential subsurface $Y,$ 
we have $\pi (r) \cap v \ne \emptyset$ 
and 
\item[(ii)] there is a geodesic ray $r \subset U$ ending at $x$ 
such that $\pi (r) \subset Y$, 
\end{itemize}
we call such a point $x$ a {\it filling point} for $Y$ (or simply, $x$ {\it fills} $Y$).
We set 
$$
{\mathbb A} =\{ x \in \partial U ~|~x ~\text{fills}~ S\}.
$$

We have the following result.

\begin{thm}[\cite{lms}, Theorem 1.1 and 3.6]\label{extension}
For any $v \in {\mathcal C}(S)$, the map
$$
\Phi(v, \cdot) : U \rightarrow {\mathcal C}(\dot{S})
$$
can be continuously extended to 
$$
\overline{\Phi}(v, \cdot) : U \cup {\mathbb A} \rightarrow \overline{\mathcal C}(\dot{S}).
$$
Moreover for every $z_{\infty} \in {\mathbb A}$, 
$\overline{\Phi}(v, z_{\infty})$ does not depend on $v$.  
\end{thm}
 
\section{Main Theorem}

Let $\gamma$ be a nontrivial simple closed curve 
on a Riemann surface $R$. 
Denote by ${\rm Mod}(A)$ the modulus of an annulus in $R$ 
whose core curve is homotopic in $R$ to $\gamma$. 
We define the extremal length ${\rm Ext}( \gamma )$ of $\gamma$ on $R$ by
$$
{\rm Ext}_{R}( \gamma ) = \inf_{A} 1/{\rm Mod}(A),
$$
where the infimum is over all annuli $A \subset R$ 
whose core curve is homotopic in $R$ to $\gamma$ 
(cf. Chapter 4 of Ahlfors \cite{ahlfors}).    

Given any point $p=[R, f] \in T(S)$ and a nontrivial simple closed curve $\alpha$ on $S$,
we define the extremal length ${\rm Ext}_{p}( \alpha )$ by
$$
{\rm Ext}_{p}( \alpha )={\rm Ext}_{R}(f(\alpha)).
$$

\begin{thm}\label{main}
The map $\varphi \circ r : T(S) \times U \rightarrow T(\dot{S})$ has a continuous extension to   
$T(S) \times {\mathbb A}$. 
\end{thm}

{\it Proof.}~~Let $\{ ( p_m, z_m )\}_{m=1}^{\infty}$ be any sequence 
in $T(S) \times U$ converging to $(p_{\infty}, z_{\infty}) \in T(S) \times {\mathbb A}$. 
Put $q_m = \varphi \circ r(p_m, z_m)$. 
We regard $\{ q_m \}_{m=1}^{\infty}$ as the sequence $\{ (q_m, \overline{q}_0) \}_{m=1}^{\infty}$ 
in a Bers slice of $T(\dot{S}) \times T( \overline{\dot{S} }) $ 
where $q_0$ is the base point $(\dot{S},i d)$ of $T(\dot{S})$. 

For each pair $(q_m,\overline{q}_0) \in T(\dot{S}) \times \{ \overline{q}_0 \}$, 
there is a unique quasifuchsian group $\Gamma_m$ 
up to conjugation 
such that it uniformizes $(q_m,\overline{q}_0)$.  
For each $\Gamma_m$, 
the quotient space $N_m={\mathbb H}^3/\Gamma_m$ is a hyperbolic manifold,  
where ${\mathbb H}^3$ is upper half space.

To prove that $\{ q_m \}_{m=1}^{\infty}$ converges, we need the following lemma.
\begin{lem}\label{lem-1}
Given $z_{\infty} \in {\mathbb A}$,  
there exists a filling lamination $\lambda$ with the following property. 
For any sequence $\{ z_m \}_{m=1}^{\infty}$ with 
$\lim_{m \to \infty} z_m = z_{\infty}$ and $q_m =\varphi \circ r(p_m, z_m)$ 
as above, there exists a sequence of simple closed curves 
$\{ \alpha_m \}_{m=1}^{\infty}$ with the following properties:
\begin{itemize}
\item[(1)] The lengths $\ell_{ N_m }( \alpha_m )$ of $\alpha_m$ in $N_m$ are bounded, and
\item[(2)] the sequence $\{ \alpha_m \}_{m=1}^{\infty}$ converges to $\lambda$ in $\overline{\mathcal{C}}(\dot{S})$. 
\end{itemize}
\end{lem}

\begin{proof}[Proof of Lemma \ref{lem-1}]
First we pick any simple closed curve $\alpha$ on $S$ and fix it. 
By Theorem \ref{extension},  
${\Phi}(\alpha, z_m) \rightarrow \lambda$ as 
$m \rightarrow \infty$ in $\overline{\mathcal C}(\dot{S})$ and 
$\lambda$ does not depend on $\alpha$.

Next we produce a sequence of curves which satisfies (1) and (2) as follows.
Let $S_m$ be the underlying Riemann surface for $p_m$ 
and $\hat{h}_m$ the Teichm\"uller map from $S$ onto $S_m$.  
Then $p_m=( S_m, \hat{h}_m)$.  
Take $\{ f_m \}_{m=1}^{\infty} \subset {\rm Diff}_0(S)$ with
$\tilde{\rm ev}(f_m)=z_m$.
Then the point $[S_m - \{ \hat{h}_m(\hat{z}_m) \}, \hat{h}_m \circ f_m]$ 
represents $q_m$ in $T(\dot{S})$ 
where $\hat{z}_m$ is the image in $S$ of $z_m$ via the projection $U \rightarrow S$. 
We choose $\alpha_m$ to be $\tilde{\Phi}(\alpha, f_m)$ 
if $\hat{z}_m$ is not contained 
in $N^{\circ}( \alpha )$, 
and 
otherwise let $\alpha_m$ be a vertex of $\tilde{\Phi}(\alpha, f_m)$ 
with weight at least $1/2$ 
in barycentric coordinates on the edge of $\tilde{\Phi}(\alpha, f_m)$ (cf. (\ref{tilde-Phi})).   

We show that the sequence $\{ \alpha_m \}_{m=1}^{\infty}$ satisfies (1) and (2).
By Theorem \ref{extension},  
$\tilde{\Phi}(\alpha, f_m)={\Phi}(\alpha, z_m) \rightarrow \lambda$ as $m \rightarrow \infty$ in  
$\overline{\mathcal C}(\dot{S})$, which implies (2).

To see (1), 
first we set
$$
E_0=1/{\rm Mod}( N( \alpha ) ). 
$$
Suppose that $\hat{z}_m=f_m(\hat{z}_0) \not\in N^{\circ}( \alpha )$.
Then the interior of the annulus $N( \alpha )$ is embedded in $S - \{ \hat{z}_m \}$.
Let $p_0$ be the besepoint of $T(S)$. 
Since $\{ d_T( p_m, p_\infty )\}_{m=1}^{\infty}$ is a bounded sequence, 
by using the triangle inequality 
we see that 
$\{ d_T( p_m, p_0 ) \} _{m=1}^{\infty}$ is also a bounded sequence. 
Hence we may assume that $K(\hat{h}_m) <K$ for every $m$ with a sufficiently large $K(>1)$.  
Since every $\hat{h}_m$ satisfies 
$$
{\rm Mod}( \hat{h}_m( N( \alpha ) ) )\geqq 1/(KE_0), 
$$
we obtain
\begin{equation}
\label{eq:Ext1}
{\rm Ext}_{ q_m }( \alpha_m ) \leqq KE_0.
\end{equation}

Suppose $\hat{z}_m\in N^\circ( \alpha )$.
Let $\alpha^*$ be the core geodesic of $N( \alpha )$ 
and denote by $\alpha^{\pm}$ the components of $\partial N( \alpha )$. 
Take a conformal (not isometric) coordinates
$$
g_m:\alpha^*\times [-\epsilon( \alpha ),\epsilon( \alpha )]\to N( \alpha )
$$
such that $\alpha^* \times \{ 0 \}$ maps to the core geodesic of $N( \alpha )$
and for each $t$,
$\alpha^* \times \{ t \}$ is sent to the equidistant circle to the core geodesic.
Let $t_m \in [-\epsilon( \alpha ), \epsilon( \alpha ) ]$ such that 
$\hat{z}_m \in g_m( \alpha^* \times \{ t_m \})$.
We suppose $t_m> 0$.
The case $t_m\leqq 0$ can be dealt with the same manner.

Let $A_m$ be the component of $N( \alpha ) \setminus g_m( \alpha^* \times \{t_m\} )$
which is containing $\alpha^*$.
Since $g_m$ is conformal,
$$
{\rm Mod}(A_m) \geqq {\rm Mod}( N( \alpha ) )/2.
$$
Thus 
$$
{\rm Mod}( \hat{h}_m(A_m) ) \geqq 1/(2KE_0).
$$
By the definition of $\alpha_m$, we have 
\begin{equation}
\label{eq:Ext2}
{\rm Ext}_{ q_m }(\alpha_m)={\rm Ext}_{ q_m }( f_m^{-1}(\alpha^{-}) ) \leqq 2KE_0.
\end{equation}

From \eqref{eq:Ext1} and \eqref{eq:Ext2}, 
we conclude that ${\rm Ext}_{q_m}(\alpha_m)$ are bounded above. 
By Maskit's comparizon theorem of \cite{Maskit2}, 
we see that $\ell_{q_m}(\alpha_m)$ are bounded above. 
Here for any point $q=[\dot{R}, \dot{f}] \in T(\dot{S})$ and a nontrivial simple closed curve $\gamma$ on $\dot{S}$
the symbol $\ell_q(\gamma)$ means the length of the geodesic representative of the homotopy class of $\dot{f}(\gamma)$ in the hyperbolic metric on $\dot{R}$. 
Therefore by Bers inequality, we have
$$
\ell_{N_m}( \alpha_m ) \leqq 2\text{min}\{ \ell_{q_m}( \alpha_m), \ell_{q_0}(\alpha_m) \}, 
$$  
and hence $\ell_{N_m}( \alpha_m )$ are uniformly bounded, which implies (1).     \hfill  \bsquare
\end{proof}

\medskip
We now return to the proof of Theorem \ref{main}.
Consider the normalized sequence $\{ \alpha_m/\ell_{ q_0 }(\alpha_m)  \}_{m=1}^{\infty}$. 
This sequence has a convergent subsequence 
(represented by the same indices) to a measured lamination $\nu$, 
which by Theorem 1.4 of \cite{klarreich} has the same support as $\lambda$ 
from Lemma \ref{lem-1} (2). 

For a hyperbolic manifold $N$ with marked homotopy equivalence
$\dot{S}\to N$, and a measured lamination $\xi$ on $\dot{S}$, 
we denote by
$\underline{\ell}_N(\xi)$ the extended length of $\xi$ in $N$ 
(see Brock \cite{brock}). 
Any
quasifuchsian group uniformizing $(q_m,\overline{q}_0)$ admits a natural marked
homotopy equivalence inherited from that of $q_m$.
By Brock's continuity theorem we get 
$$
\underline{\ell}_{N_m}\left( \frac{ \alpha_m }{ \ell_{ q_0 }( \alpha_m) } \right) \rightarrow 
\underline{\ell}_{N\infty}(\nu) 
~~\text{as}~~m \rightarrow \infty
$$
where $N_{\infty}={\mathbb H}^3/\Gamma_{\infty}$ is a marked hyperbolic manifold  
and $\Gamma_{\infty}$ is an algebraic limit 
of the subsequence 
$\{ \Gamma_m \}_{m=1}^{\infty}$.                                                               
(cf. Theorem 2 of \cite{brock}. 
See also Lemma 3.1 of Ohshika \cite{ohshika}). 
On the other hand, from (2) of Lemma \ref{lem-1}, 
because $\alpha_m$ tends to infinity in $\mathcal{C}(\dot{S})$, 
in the fixed metric $q_0$, we must have $\ell_{q_0}(\alpha_m) \to \infty$ 
as $m\to \infty$. 
Therefore, from (1) in Lemma \ref{lem-1}, we have
\begin{eqnarray*}
\underline{\ell}_{N_m}\left( \frac{ \alpha_m}{ \ell_{ q_0 }(\alpha_m) } \right)  & = & 
\frac{1}{ \ell_{ q_0 }(\alpha_m) }  \underline{\ell}_{N_m}( \alpha_m ) \\ 
                                                                             &\rightarrow & 0 ~~~(m \rightarrow \infty), 
\end{eqnarray*}
and thus the length of $\nu$ in $N_{\infty}$ is zero. 
Since the support of $\nu$ contains  $\lambda$ as its support, 
the length of $\lambda$ in $N_{\infty}$ is also zero. 
Hence $\lambda$ is not realizable in $N_{\infty}$. 
Since $\lambda$ is filling, 
it follows $\Gamma_{\infty}$ is a singly degenerate Kleinian group. 
By using Ending lamination theorem for surface groups of \cite{bcm}, 
$\Gamma_{\infty}$ is uniquely determined by $(\lambda, \overline{q}_{0})$ up to conjugation.  
By Theorem \ref{extension}, $\lambda$ depends only on $z_{\infty}$. 
Thus the sequence $\{ q_m \}_{m=1}^{\infty}$ converges 
without depending on the choice of a convergent sequence to $(p_{\infty}, z_{\infty}) \in T(S)  \times {\mathbb A}$.  
Hence we conclude that the map $\varphi \circ r : T(S) \times U \rightarrow T(\dot{S})$ has a continuous extension  
to $T(S) \times {\mathbb A}$. 
\hfill  \bsquare

\vspace{3mm}

\begin{center}
ACKNOWLEDGMENT
\end{center}
The authors wish to thank to the referee for valuable suggestions on the improvement of our proof of Main theorem.


\bibliographystyle{amsplain}

\end{document}